\documentclass[11pt,a4paper]{amsart}
\usepackage{amsmath,amssymb}
\newtheorem{theorem}{Theorem}
\newtheorem{proposition}{Proposition}
\newtheorem{lemma}{Lemma}

\begin{document}
\title{A characterisation of virtually free groups}

\author[R.~Gilman]{Robert~H.~Gilman}
\address{Department of Mathematics\\
Stevens Institute of Technology\\
Hoboken NJ 07030, USA}
\email{rgilman@stevens-tech.edu}

\author[S.~Hermiller]{Susan Hermiller}
\address{Department of Mathematics\\
        University of Nebraska\\
         Lincoln NE 68588-0130, USA}
\email{smh@math.unl.edu}

\author[Derek F.~Holt]{Derek F.~Holt}
\address{Mathematics Institute\\
      University of Warwick  \\
       Coventry CV4 7AL, UK  }
\email{dfh@maths.warwick.ac.uk}

\author[Sarah Rees]{Sarah Rees}
\address{School of Mathematics and Statistics\\
       University of Newcastle \\
       Newcastle NE1 7RU, UK  }
\email{Sarah.Rees@ncl.ac.uk}

\maketitle

\begin{abstract}
We prove that a finitely generated group $G$ is virtually free if and only if
there exists a generating set for $G$ and $k > 0$ such that all $k$-locally
geodesic words with respect to that generating set are geodesic.
\end{abstract}

\begin{center}
{\bf Keywords:} Virtually free group; Dehn algorithm; word problem.

{\bf Mathematics Subject Classification:} 20E06; Secondary 20F67.
\end{center}

\section{Introduction}
A group is called virtually free if it has a free subgroup of finite index.

In this article we characterise finitely generated virtually free
groups by the property that a Dehn algorithm reduces any word
to geodesic form. Equivalently, a group is virtually free precisely when
the set of $k$-locally geodesic words and the set of geodesic words coincide
for suitable $k$ and appropriate generating set.

Let $G$ be a group with finite generating set $X$.
We shall assume throughout this article that all generating sets of
groups are closed under the taking of inverses.
For a word $w=x_1\cdots x_n$ over $X$,
we define $l(w)$ to be the length $n$ of $w$ as a string,
and $l_G(w)$ to be the length of the shortest word representing the same
element as $w$ in $G$.
Then $w$ is called a {\em geodesic} if $l(w)=l_G(w)$, and 
a {\em $k$-local geodesic} if every subword of $w$ of length at most 
$k$ is geodesic.

Let $\mathcal{R}$ be a finite set of length-reducing rewrite rules
for $G$; that is, a set of substitutions
\[ u_1\rightarrow v_1,
u_2\rightarrow v_2, \ldots, u_r \rightarrow v_r, \]
where $u_i =_G v_i$ and $l(v_i) < l(u_i)$ for $1 \le i \le r$.
Then $\mathcal{R}$ is called a {\em Dehn algorithm} for $G$ over $X$
if repeated application of these rules reduces any representative of the
identity to the empty word.
It is well-known that a group has a Dehn algorithm if and only
if it is word-hyperbolic \cite{Alonso}.

More generally (that is, even outside of the group theoretical context),
if $L$ is any set of strings over an alphabet $X$ (or, in
other words, $L$ is any {\em language} over $X$), we shall call
$L$ {\em $k$-locally excluding} if there exists a finite set $F$ of strings of
length at most $k$ such that a string $w$ over $X$ is in $L$ if and
only if $w$ contains no substring in $F$.
It is clear that the set of $k$-local geodesics in a 
group is $k$-locally excluding, since we can choose $F$ to be the set of
all non-geodesic words of length at most $k$.
We observe in passing that if a set of strings is $k$-locally excluding
then, by definition, it is a $k$-locally testable and hence locally testable 
language (see \cite{Pin}).

We shall say that the group $G$ is $k$-locally excluding over a finite
generating set $X$ when
the set of geodesics of $G$ over $X$ is $k$-locally excluding.

The purpose of this paper is to prove the following theorem.
\begin{theorem}\label{main}
Let $G$ be a finitely generated group. Then the following are equivalent.
\begin{enumerate}
\item[(i)] $G$ is virtually free.
\item[(ii)] There exists a finite generating set $X$ for $G$ and a
finite set of length-reducing rewrite rules over $X$ whose application
reduces any word over $X$ to a geodesic word; 
that is $G$ has a Dehn algorithm that reduces all words to geodesics.
\item[(iii)] There exists a finite generating set $X$ for $G$ and an integer
$k$ such that every $k$-locally geodesic word over $X$ is a geodesic;
that is, $G$ is $k$-locally excluding over $X$.
\end{enumerate}
\end{theorem}

\section{Proof of Theorem~\ref{main}}
The equivalence of (ii) and (iii) is straightforward.
Assume (ii), and let $\mathcal{R}$ be a set of length-reducing rewrite rules
with the specified property. Let $k$ be the maximal length of a left hand
side of a rule in $\mathcal{R}$.  Then a $k$-local geodesic over $X$ cannot
have the left hand side of any rule in $\mathcal{R}$ as a subword, and
so it must be geodesic. Conversely, assume (iii) and let $\mathcal{R}$ 
be the set of all rules $u \rightarrow v$ in which $l(v) < l(u) \le k$
and $u =_G v$. Then repeated application of rules in $\mathcal{R}$ reduces
any word to a $k$-local geodesic which, by (iii), is a geodesic.

The main part of the proof consists in showing that (i) and (iii) are
equivalent.  We start with a useful lemma.

\begin{lemma}\label{suffixreduction}
Let $G$ be a group with finite generating set $X$, let $k>0$ be an integer,
and suppose that $G$ is $k$-locally excluding over $X$.
Let $w$ be a geodesic word over $X$, and let $x \in X$. Then\\
(i) $l_G(wx)$ is equal to one of $l(w)+1$, $l(w)$, $l(w)-1$.\\
(ii) $wx$ is geodesic (that is, $l_G(wx)=l(w)+1$) if and only if $vx$
is geodesic, where $v$ is the suffix of $w$ of length
$k-1$ (or the whole of $w$ if $l(w)<k-1$).\\
(iii) $l_G(wx)-l(w)=l_G(v'x)-l(v')$, where $v'$ is the suffix of 
$w$ of length $2k-2$ (or the whole of $w$ if $l(w)<2k-2$).\\
\end{lemma}
\begin{proof}
The three possibilities for $l_G(wx)$ follow from the fact that $w$
is geodesic and $x$ is a single generator.
(ii) is an immediate consequence of $G$ being $k$-locally excluding.
(iii) follows from (ii) when $wx$ is geodesic, so suppose not.
Write $w=uv$ with $v$ as defined in (ii),
and let $z$ be a geodesic representative of $vx$.
Since $v$ is geodesic, $l(z)$ is either $l(v)$ or $l(v)-1$.
In the second case $uz$ is geodesic, so $l_G(wx)-l(w)=l_G(vx)-l(v)
=l_G(v'x)-l(v')= -1$ and (iii) follows.
In the first case ($l(z)=l(v)$)  write $w=u'v''v$ with $v'=v''v$,
so $l(v'')= k-1$ provided that $u'$ is non-empty.
Now $wx = u'v''vx =_G u'v''z$ where $l(u'v''z) = l(w)$,
and either $l_G(wx) = l(u'v''z) = l(w)$ or $l_G(wx) = l(u'v''z)-1 = l(w)-1$. 
So at most one length reduction occurs in the word $u'v''z$,
and since $u'v''$ is geodesic, that length reduction must occur, if at all,
within the subword $v''z =_G v'x$.
Part (iii) follows from this.
\end{proof}

We are now ready to prove that (iii) implies (i) in Theorem~\ref{main}.

\begin{proposition} Suppose that $G$ is a group with finite generating set $X$
and that the geodesics over $X$ are  $k$-locally excluding for some $k>0$.
Then $G$ is virtually free.
\end{proposition}
\begin{proof}
We prove this result by demonstrating that the word problem for $G$
can be solved on a pushdown automaton,
and then using Muller and Schupp's classification of groups with
this property~\cite{MullerSchupp}.

The automaton to solve the word problem operates as follows.
Given an input word $w$, the automaton reads $w$ from left to right.
At any point, the word on the stack
is a geodesic representative of the word read so far. 
Suppose at some point it has $u$ on the stack and then reads a symbol $x$.
It pops $2k-2$ symbols off the stack (or the whole of $u$ if $l(u)<2k-2$),
appends $x$ to the end of the word so obtained,
replaces it by a geodesic representative if necessary, and 
appends that reduced word to the stack.
It follows from Lemma~\ref{suffixreduction} that the word now on the stack
is a geodesic representative of $ux$, and hence of the word read so far.

So $w$ represents the identity in $G$ if and only if the stack is empty once
all the input has been read and processed, and it follows
immediately from~\cite{MullerSchupp} that $G$ is virtually free.
\end{proof}

It remains to prove that (i) implies (iii),
namely that the set of geodesics of a virtually free group with an
appropriate generating set is $k$-locally excluding for some $k>0$.

It is proved in~\cite[Theorem 7.3]{ScottWall}
that a finitely generated group $G$ is virtually free
if and only if it arises as follows:
$G$ is the fundamental group of a graph of groups $\Gamma$ with
finite vertex groups $G_1,\ldots G_n$, and finite edge groups $G_{i,j}$
for certain pairs $\{i,j\}$.

There are various alternative and equivalent definitions of
the fundamental group of a graph of groups, but the one that is most
convenient for us is~\cite[Chapter 1, Definition 3.4]{DD}.
As is pointed out in~\cite[Chapter 1, Example 3.5 (vi)]{DD},
such a group $G$ can be built up as a sequence of groups
$1=H_1,H_2, \ldots, H_r = G$, where each $H_{i+1}$ is defined either
as a free product with amalgamation (over an edge group)
of $H_i$ with one of the vertex groups $G_i$,
or as an HNN extension of $H_i$ with associated subgroups isomorphic to
one of the edge groups $G_{i,j}$.
The amalgamated free products are done first, building
up along a maximal tree, and then the HNN extensions are
done for the remaining edges in the graph.

So from now on we shall assume that our virtually free group $G$
can be constructed in this way, where the groups $G_i$ and
$G_{i,j}$ are all finite.
Hence the result follows from repeated application of the
following two lemmas, of which the proofs are very similar.

Notice that the generating set $X$ over which $G$ is $k$-locally
excluding will contain all non-identity elements of each of the vertex groups,
$G_i$ and also certain other elements arising from the HNN extensions,
which are specified in Lemma~\ref{lemma2}.

\begin{lemma}
\label{lemma1}
Let $H$ be a group which is $k$-locally excluding over a
generating set $X$ for some $k \geq 2$, let $K$ be a finite group,
let $A = H \cap K$,
and suppose that  $A\setminus \{1\} \subset X$.

Then $G=H *_A K$ is $k'$-locally excluding over
$X' :=X \cup (K\setminus A)$, where $k'= 3k-2$.
\end{lemma}

\begin{lemma}\label{lemma2}
Let $H$ be a group which is $k$-locally excluding over a generating set $X$
for some $k \geq 2$,
let $A$ and $B$ be isomorphic finite subgroups of $H$ which satisfy
$A\setminus \{1\} \subset X$ and $B\setminus \{1\} \subset X$, and let $G =
\langle H,t \rangle$ be the HNN extension in which $tat^{-1} = \phi(a)$ for
all $a \in A$, where $\phi:A \to B$ is an isomorphism.

Then $G$ is $k'$-locally excluding over
$X' := X \cup\, \{ ta \mid a \in A \}\cup\, \{ t^{-1}b \mid b \in B \}$,
where $k'= 3k-2$.
(Note that the elements of $X'$ in the set $\{ t^{-1}b \mid b \in B \}$ are
the inverses of those in the set $\{ ta \mid a \in A \}$.)
\end{lemma}

\begin{proof}[Proof of Lemma~\ref{lemma1}]
Let $w$ be a $k'$-local geodesic of $G$ over $X'$.
We want to prove that $w$ is geodesic.
Suppose not, and let $w'$ be a geodesic word that represents the
same element of $G$.
Note that, since $A \setminus \{1\} \subseteq X'$, we cannot have
$w \in A$, because that would imply that $l(w) \le 1$.

We can write $w  = w_0k_1w_1k_2\cdots k_rw_r$, where
each $k_i \in K \setminus A$ and each $w_i \in X^*$.
Either $w_0$ or $w_r$ could be the empty word but,
since $K \setminus \{1\} \subseteq X'$ and $w$ is a $k'$-local geodesic with
$k' > k \ge 2$, $w_i$ must be non-empty for $0 < i < r$.  The $2$-locally
excluding condition also implies that no non-empty
$w_i$ is a word in $A^*$.
In fact, since $H$ is by assumption $k$-locally excluding over $X$
and $k' > k$, the words $w_i$ are geodesics as elements of $H$ over $X$, and
so the non-empty $w_i$ represent elements of $H \setminus A$.

Similarly, write $w' = w'_0k'_1w'_1k'_2\cdots k'_{r'}w'_{r'}$.

Now the normal form theorem for free products with
amalgamation (see~\cite[Thm 4.4]{MKS} or the remark
following~\cite[Chapter 4, Theorem 2.6]{LyndonSchupp})
states that, if $C$ is a union of sets of
distinct right coset representatives of $A$ in $H$ and in $K$,
then any element of the amalgamated
product can be written uniquely as a product of the form
$ac_1\cdots c_s$, where $a \in A$, each $c_i \in C$,
and alternate $c_i$'s are in $H \setminus A$ and $K \setminus A$.

Since each $k_i \in K \setminus A$ and each non-empty
$w_i \in H \setminus A$, the syllable length $s$ of the group element
represented by $w$ is equal to the number of non-trivial words
$w_0$, $k_1$, $w_1, \ldots, k_r$, $w_r$,
where $c_1 \in H \setminus A$ if and only if $w_0$ is
non-trivial, and $c_s \in H \setminus A$ if and only if
$w_r$ is non-trivial. The same applies to $w'$, and hence
$r=r'$, $w_0$ and $w'_0$ are either
both empty or both non-empty, and similarly for $w_r$ and $w'_r$.

Furthermore, $w_r$ and $w_r'$ are in the same right coset of $A$ in
$H$, and so $w_r' =_H a_rw_r$ for some $a_r \in A$. Then
$k_r$ and $k_r'a_r$ are in the same right coset of $A$ in $K$,
and so $k_r =_K b_{r-1} k_r' a_r$ for some $b_{r-1} \in A$. 
Carrying on in this manner, we can show that there exist
$a_i,b_i \in A$ ($0 \le i \le r$) such that
$w'_i =_H a_iw_ib_i$ and $k'_i =_K b_{i-1}^{-1}k_ia_{i}^{-1}$,
where $a_0 = b_r = 1$.


Since $r = r'$ and $l(w') < l(w)$, we must have $l(w'_i) < l(w_i)$
for some $i$.
So one of the words $a_iw_i$, $w_ib_i$, $a_iw_ib_i$ must reduce
(in $H$ over $X$) to a word strictly shorter than $w_i$.

Suppose first that $w_ib_i$ reduces to a word strictly shorter than $w_i$.
Since $b_r = 1$, we have $i<r$ and so $k_{i+1}$ exists.
Then, by Lemma~\ref{suffixreduction}, $l_H(v'_ib_i) = l(v'_i) - 1$,
where $v'_i$ is the suffix of $w_i$ of length $2k-2$,
or the whole of $w_i$ if $l(w_i) < 2k-2$.
Now, since $v_i' k_{i+1} = _G (v_i' b_i)(b_i^{-1}k_{i+1})$
with $b_i^{-1}k_{i+1} \in K$, we see that the suffix $v_i'k_{i+1}$ of
$w_ik_{i+1}$, which has length at most $2k-1$,
is a non-geodesic word in $G$ and, since $2k-1 < k'$, this
contradicts the assumption that $w$ is a $k'$-local geodesic.
 
The case in which $a_iw_i$ reduces to a word of length less than $w_i$ is
similar (here we use a `mirror image' of Lemma~\ref{suffixreduction}),
and we find that $i>0$ and a prefix of $k_iw_i$ of length at most $2k-1$ is
non-geodesic, again contradicting the assumption that $w$ is a $k'$-local
geodesic.

It remains to consider the case where the reduction (in $H$ over $X$)
of $a_iw_ib_i$ is strictly shorter than $w_i$, but each of the reductions of
$a_iw_i$ and $w_ib_i$ have the same length as $w_i$.
Since neither $a_i$ nor $b_i$ can be trivial, we have $0 < i < r$, and
so $k_i$ and $k_{i+1}$ both exist.
We claim that $w_i$ has length at most $3k-4$.
For if not, we write $w_i = u'uv'$, where $l(u') = l(v') = k-1$ and $l(u)
\ge k-1$, and deduce from Lemma~\ref{suffixreduction} and its mirror image
that $a_iw_ib_i =_H yuz$, where $y,z \in X^*$ and $l(y)=l(z)=k-1$.
Then since $yuz$ reduces in H over $X$ and $H$ is $k$-locally excluding over
$X$, some subword of length $k$ must reduce. Such a subword must be a subword of
either $yu$ or $uz$, and so one of $a_iw_i$ or $w_ib_i$ does indeed reduce to
a word shorter than $w_i$, contradicting our assumption.
Hence $l(w_i) \le 3k-4$ as claimed.

Now $k_iw_ik_{i+1}$ has length $2 + l(w_i) \le 3k-2$, but
$k_iw_ik_{i+1} =_G (k_ia_i^{-1})w_i'(b_i^{-1}k_{i+1})$
with $k_ia_i^{-1}, b_i^{-1}k_{i+1} \in K$,
so $k_iw_ik_{i+1}$ is not a geodesic in $G$ over $X'$,
and once again we contradict our assumption that $w$ is a $k'$-local geodesic.
This completes the proof of Lemma~\ref{lemma1}.
\end{proof}

\begin{proof}[Proof of Lemma~\ref{lemma2}]
Let $w$ be a $k'$-local geodesic of $G$ over $X'$.
We want to prove that $w$ is geodesic.
Suppose not, and let $w'$ be a geodesic word that represents the
same element of $G$.

Write
$w  = w_0t_1^{\epsilon_1}w_1t_2^{\epsilon_2}w_2\cdots t_r^{\epsilon_r}w_r$,
where each $t_i$ is one of the generators of the form $t a$ ($a \in A)$, each
$\epsilon_i$ is $1$ or $-1$, and each $w_i$ is a word over $X$.
Since $k'>k$, $w$ is a $k$-local geodesic, so each word $w_i$ is geodesic as an
element of $H$.
So if $w_i$ represents a non-trivial  element of $A$ or of $B$,
then $w_i$ has length 1.  Hence, if $\epsilon_i = 1$ then we cannot have
$w_i \in A \setminus \{1\}$, and if $\epsilon_i = -1$ then we cannot have
$w_i \in B \setminus \{1\}$, because in those cases
$t^{\epsilon_i} w_i$ would be a non-geodesic subword of $w$ of length 2.
Also, if $w_i$ is empty with $0 < i < r$, then $\epsilon_i = \epsilon_{i+1}$.

Similarly, write $w'  = w_0'(t_1')^{\epsilon_1'}w_1'(t_2')^{\epsilon_2'}
w_2'\cdots (t'_{r'})^{\epsilon'_{r'}}w'_{r'}$.

Now the normal form theorem for HNN extensions
\cite[Chapter 4, Theorem 2.1]{LyndonSchupp}
states that if $C$ is a union of sets $H_A$ and $H_B$ of distinct
right coset representatives of $A$ and of $B$ in $H$, then any
element of the HNN extension $G$ can be written uniquely as a product of the
form $ht^{\varepsilon_1}c_1 \cdots t^{\varepsilon_s}c_s $,
where $h \in H$, each $\varepsilon_i$ is $1$ or $-1$,
each $c_i \in C$, and $c_i \in H_A$ or $c_i \in H_B$
when $\varepsilon_i = 1$ or $-1$, respectively.
Also, if $c_i = 1$ with $1 \le i < s$, then $\varepsilon_i=\varepsilon_{i+1}$.

For the normal form of the element of $G$ represented by both $w$ and $w'$,
it follows that $r=r' = s$ and  $\epsilon_i= \epsilon_i' =
\varepsilon_i$ for each $i$. Furthermore,
an inductive argument similar to the one in the proof of
Lemma~\ref{lemma1} shows that there are elements
$a_i,b_i \in A \cup B$ ($0 \le i \le r$) such that $w'_i =_H a_iw_ib_i$ and
$(t'_i)^{\epsilon_i} = b_{i-1}^{-1}(t_i)^{\epsilon_i}a_{i}^{-1}$,
where $a_0 = b_r = 1$.
We have $a_i \in A$ or $B$ when $\epsilon_i = 1$ or $-1$, respectively, and
$b_i \in B$ or $A$ when $\epsilon_{i+1} = 1$ or $-1$, respectively.


Since $r = r'$ and $l(w') < l(w)$, we must have $l(w'_i) < l(w_i)$
for some $i$.
So one of the words $a_iw_i$, $w_ib_i$, $a_iw_ib_i$ must reduce
(in $H$ over $X$) to a word strictly shorter than $w_i$.

Suppose first that $w_ib_i$ reduces to a word strictly shorter than $w_i$.
Since $b_r = 1$, we have $i<r$ and so $t_{i+1}$ exists.
Then, by Lemma~\ref{suffixreduction}, $l_H(v'_ib_i) = l(v'_i) - 1$,
where $v_i'$ is the suffix of $w_i$ of length $2k-2$,
or the whole of $w_i$ if $l(w_i) < 2k-2$.
Now, since $v_i' t_{i+1}^{\epsilon_{i+1}} =_G
(v_i' b_i)(b_i^{-1}t_{i+1}^{\epsilon_{i+1}})$ with
$b_i^{-1}t_{i+1}^{\epsilon_{i+1}} \in X'$, we see that the suffix
$v_i't_{i+1}^{\epsilon_{i+1}}$
of $w_it_{i+1}^{\epsilon_{i+1}}$, which has length
at most $2k-1$, is a non-geodesic word in $G$ and, since $2k-1 < k'$, this
contradicts the assumption that $w$ is a $k'$-local geodesic.
 
The case in which $a_iw_i$ reduces to a word of length less than $w_i$ is
similar (using the mirror image of Lemma~\ref{suffixreduction}), and
we find that $i>0$ and a prefix of $t_i^{\epsilon_i}w_i$ of length at most
$2k-1$ is non-geodesic, again contradicting the assumption that $w$ is a
$k'$-local geodesic.

It remains to consider the case where the reduction (in $H$ over $X$)
of $a_iw_ib_i$ is strictly shorter than $w_i$, but each of the reductions of
$a_iw_i$ and $w_ib_i$ have the same length as $w_i$.
Since neither $a_i$ nor $b_i$ can be trivial, we have $0 < i < r$, and
so $t_i$ and $t_{i+1}$ both exist.
We claim that $w_i$ has length at most $3k-4$.
For if not, we write $w_i = u'uv'$, where $l(u') = l(v') = k-1$ and $l(u) \ge
k-1$, and deduce from Lemma~\ref{suffixreduction} and its mirror image
that $a_iw_ib_i =_G yuz$, where $y,z \in X^*$ and $l(y)=l(z)=k-1$.
Then since $yuz$ reduces in $H$ over $X$ and $H$ is $k$-locally excluding
over $X$, some subword of length $k$ must reduce. Such a
subword must be a subword of either $yu$ or $uz$, and so one of $a_iw_i$ or
$w_ib_i$ does indeed reduce to a word shorter than $w_i$, contradicting our
assumption. Hence $l(w_i) \le 3k-4$ as claimed.

Now $t_i^{\epsilon_i}w_it_{i+1}^{\epsilon_{i+1}}$ has length
$2 + l(w_i) \le 3k-2$, but $t_i^{\epsilon_i}w_it_{i+1}^{\epsilon_{i+1}} =_G
(t_i^{\epsilon_i}a_i^{-1})w_i'(b_i^{-1}t_{i+1}^{\epsilon_{i+1}})$ with
$l_G(t_i^{\epsilon_i}a_i^{-1}) = l_G(b_i^{-1}t_{i+1}^{\epsilon_{i+1}}) =1$, 
so $t_i^{\epsilon_i}w_it_{i+1}^{\epsilon_{i+1}}$ is not a geodesic in $G$
over $X'$, and once again we contradict our
assumption that $w$ is a $k'$-local geodesic. This completes the proof
of Lemma~\ref{lemma2}.
\end{proof}

\end{document}